\documentclass{amsart}
\usepackage[english]{babel}
\usepackage[utf8]{inputenc}
\usepackage{amsthm, amsmath, amssymb, amsrefs, enumerate, enumitem, mathtools, nicefrac, bbm}

\binoppenalty=\maxdimen
\relpenalty=\maxdimen

\usepackage[T1]{fontenc}

\usepackage[colorlinks=true,linkcolor=black,citecolor=black,filecolor=black]{hyperref}

\theoremstyle{plain}
\newtheorem{thm}{Theorem}[section]

\newtheorem{lemma}[thm]{Lemma}

\theoremstyle{definition}
\newtheorem{defn}[thm]{Definition}

\theoremstyle{remark}
\newtheorem*{rmk}{Remark}
\newtheorem*{rmks}{Remarks}

\usepackage{xpatch}
\makeatletter
\AtBeginDocument{\xpatchcmd{\@thm}{\thm@headpunct{.}}{\thm@headpunct{}}{}{}}
\xpatchcmd{\proof}{\@addpunct{.}}{\@addpunct{:}}{}{}
\makeatother

\makeatletter
\let\@@pmod\pmod
\DeclareRobustCommand{\pmod}{\@ifstar\@pmods\@@pmod}
\def\@pmods#1{\mkern4mu({\operator@font mod}\mkern 6mu#1)}
\makeatother

\newcommand{\C}{\mathbb{C}}
\renewcommand{\H}{\mathbb{H}}
\newcommand{\Z}{\mathbb{Z}}
\newcommand{\Q}{\mathbb{Q}}
\newcommand{\N}{\mathbb{N}}
\newcommand{\R}{\mathbb{R}}
\newcommand{\slz}{{\text {\rm SL}}_2(\mathbb{Z})}
\newcommand{\bsfrac}[2]{\reflectbox{\nicefrac[\reflectbox]{$#1$}{$#2$}}}
\DeclareMathOperator{\sgn}{sgn}
\DeclareMathOperator{\tr}{tr}
\renewcommand{\Re}{\mathrm{Re}}
\renewcommand{\Im}{\mathrm{Im}}
\newcommand{\vt}[1]{\left\lvert #1 \right\rvert}
\newcommand{\Qc}{\mathcal{Q}}
\newcommand{\Ec}{\mathcal{E}}
\newcommand{\Id}{\mathbbm{1}}

\title[Eisenstein series of even weight $k\geq 2$]{Eisenstein series of even weight $k \geq 2$ and integral binary quadratic forms}
\author{Andreas Mono}
\address{Department of Mathematics and Computer Science, Division of Mathematics, University of Cologne, Weyertal 86-90, 50931 Cologne, Germany}
\email{amono@math.uni-koeln.de}
\date{} 

\begin{document}


\begin{abstract}
We prove a conjecture of Matsusaka on the analytic continuation of hyperbolic Eisenstein series in weight $2$ on the full modular group $\slz$.
\end{abstract}

\maketitle

\small
\textbf{\keywordsname :} \textit{Eisenstein series, Integral binary quadratic forms, Analytic continuation}

\normalsize

\section{Introduction and statement of results}
Integral binary quadractic forms play a decisive role in the construction of many modular objects, mainly to investigate various classes of theta functions. However, they can also be utilized to construct another prominent class of modular objects, namely families of Eisenstein series. We will define Eisenstein series associated to some $\gamma \in \slz\setminus\{\pm\Id\}$, and call them elliptic, parabolic or hyperbolic respectively corresponding to the motion $\gamma$ induces on the upper half plane $\H$. Although such constructions go back to Petersson \cite{pet44} essentially, and the analytic continuation of the classical parabolic Eisenstein series was established by Selberg \cite{sel56} and Roelcke \cite{roe67} some years later, similar results remained elusive in the other two cases.

A first breakthrough was made in weight $0$ some years ago, which completes the picture regarding analytic continuation to $s=1$, and was established by Jorgenson, Kramer, von Pippich, Schwagenscheidt, V{\"o}lz, compare \cite[Theorem 4.2]{jokrvp10}, \cite[Section 4]{pi16}, \cite[Theorem 1.2]{pischvoe17}, \cite[Appendix B]{matsu2}. Hence, it seems natural to ask whether similar results hold in weight $2$ or higher, overleaping the ``point of symmetry'' $k=1$. In a recent paper \cite{matsu2}, Matsusaka investigated parabolic, elliptic, and hyperbolic Eisenstein series in weight $2$. In a second breakthrough, Bringmann, Kane \cite{brika} provided the analytic continuation of Petersson's weight $2$ elliptic Poincar\'{e} series to $s=0$, which enabled Matsusaka \cite[Theorem 2.3]{matsu2} to extend this result to the weight $2$ elliptic Eisenstein series. 

Consequently, we focus on the case of hyperbolic Eisenstein series. If $k = 2$, Schwagenscheidt \cite[Remark 5.4.6]{schw18} argued towards existence of the analytic continuation to $s=0$, and Matsusaka \cite[eq.\ (2.12)]{matsu2} conjectured its shape. We extend Matsusaka's setting to general even weight $k \geq 2$, and embed his Eisenstein series into a framework based on discriminants of integral binary quadratic forms. This enables us to prove Matsusaka's conjecture for any positive non-square discriminant in weight $2$ by computing the Fourier expansion of our hyperbolic Eisenstein series. To this end, we adapt Zagier's method \cite[Section 2]{zagier75}, and appeal to results of Duke, Imamo\={g}lu, T\'{o}th \cite{duimto11}. 

We introduce all involved objects and terminology during sections \ref{sec:pre} to \ref*{sec:peeis} in detail.
\begin{thm} \label{thm:main}
Let $\gamma \in \slz$ be hyperbolic and primitive. Then the function $\Ec_{2,\gamma}(\tau,s)$ can be analytically continued to $s=0$ and the continuation is given by
\begin{align*}
\lim_{s \to 0} \Ec_{2,\gamma}(\tau,s) =& \frac{-2}{\Delta(\gamma)^{\frac{1}{2}}} \sum_{m \geq 0} \sum_{Q \in \Qc\left(\Delta(\gamma)\right)_{\sim}} \chi_d(Q) \int_{\bsfrac{S_Q}{\Gamma_Q}} j_m(w) \frac{\vt{dw}}{\Im{(w)}} \ q^m \\
& - \frac{-2}{\Delta(\gamma)^{\frac{1}{2}}} \tr_{d,\Delta(\gamma)}(1) E_2^*(\tau)
\end{align*}
for any $\tau \in \H$. Here, $\tr_{d,\Delta(\gamma)}(1)$ is a twisted trace of cycle integrals given by
\begin{align*}
\tr_{d,\Delta(\gamma)}(1) \coloneqq \sum_{Q \in \Qc\left(\Delta(\gamma)\right)_{\sim}} \chi_d(Q) \int_{\bsfrac{S_Q}{\Gamma_Q}} \frac{\vt{dw}}{\Im{(w)}}.
\end{align*}
Furthermore, if $\Im(\tau)$ is sufficiently large, that is $\tau$ is located above the net of geodesics $\bigcup_{Q \in \Qc(\Delta(\gamma))} S_Q$, then we have
\begin{align*}
\lim_{s \to 0} \Ec_{2,\gamma}(\tau,s) = \frac{-2}{\Delta(\gamma)^{\frac{1}{2}}} \sum_{Q \in \Qc\left(\Delta(\gamma)\right)_{\sim}} \chi_d(Q) \int_{\bsfrac{S_Q}{\Gamma_Q}} \left(\frac{D(j)(\tau)}{j(w)-j(\tau)} - E_2^*(\tau)\right)\frac{\vt{dw}}{\Im{(w)}}.
\end{align*}
\end{thm}

\begin{rmks}
\begin{enumerate}
\item The function $\Ec_{2,\gamma}(\tau,s)$ is a twisted trace of individual hyperbolic Eisenstein series, which we denote by $E_{2,\gamma}(\tau,s)$. 
\item We will indicate below that the analytic continuation of the weight $2$ parabolic / elliptic Eisenstein series to $s=0$ is a harmonic / polar harmonic Maa{\ss} form. Such forms generalize the notion of classical holomorphic modular forms by relaxing analytical and growth conditions to a non-holomorphic setting. Theorem \ref{thm:main} completes the picture in the sense that the resulting cycle integral is a locally harmonic Maa{\ss} form of weight $2$ in $\tau$ with $\Im(\tau)$ sufficiently large. These objects were introduced by Bringmann, Kane, Kohnen in \cite{brikako} (see also \cite[Section 13.4]{thebook}), and independently by Hövel \cite{hoevel} in his PhD. thesis. Roughly speaking, such a form is a harmonic Maa{\ss} form that is permitted to have singularities on the net of geodesics $\bigcup_{Q \in \Qc(D)} S_Q$. The singularities occur due to the presence of a sign-function in most cases, and are called ``jumping singularities''.
\end{enumerate}
\end{rmks}

As a byproduct of our approach, we obtain the expansion of $\Ec_{k,\gamma}(\tau,0)$ for every even weight $k \geq 4$. This was known by Parson \cite[Theorem 3.1]{parson} without the twisting. In particular, if $k \geq 4$ satisfies $k \equiv 0 \pmod*{4}$, the hyperbolic Eisenstein series $E_{k,\gamma}(\tau,0)$ is a holomorphic cusp form. In this case, the Fourier expansion of the twisted traces of hyperbolic Eisenstein series of weight $4 \mid k > 2$ was already established by Gross, Kohnen, Zagier \cite[p.\ 517]{grokoza}.
\begin{thm} \label{thm:var}
Let $\gamma \in \slz$ be hyperbolic and primitive, and suppose $k \geq 4$ is even. Moreover, let $G_m(\tau,s)$ be the Niebur Poincar\'{e} series studied by Duke, Imamo\={g}lu, T\'{o}th in \cite{duimto11}. Then $\Ec_{k,\gamma}(\tau, 0)$ equals
\begin{align*}
\frac{(-1)^{\frac{k}{2}}2 \pi^{\frac{k}{2}}}{\Delta(\gamma)^{\frac{k}{4}}\Gamma\left(\frac{k}{4}\right)^2} \sum_{m \geq 1} m^{\frac{k}{2}-1} \sum_{Q \in \Qc\left(\Delta(\gamma)\right)_{\sim}} \chi_d(Q) \int_{\bsfrac{S_Q}{\Gamma_Q}} G_{-m}\left(w,\frac{k}{2}\right) \frac{\vt{dw}}{\Im{(w)}} \ q^m.
\end{align*}
\end{thm}

We devote Section \ref{sec:hypeis} to the development of both theorems. 

\subsection*{Acknowledgements:} 
The author would like to thank his PhD-advisor Kathrin Bringmann for suggesting Matsusaka's paper as well as for her continuous helpful feedback to the work on it. Furthermore, the author would like to thank the anonymous refree for many valuable and significant comments on an earlier version, which improved the paper vigorously.

\section{Preliminaries} \label{sec:pre}
Let us summarize some general framework first, more details can be found for example in \cite[Chapter 2]{iwaniec97} regarding hyperbolic geometry, and in \cite[§ 8]{zagier81} regarding integral binary quadratic forms.

\subsection{Fractional linear transformations}
Let $\gamma = \left(\begin{smallmatrix} a & b \\ c & d \end{smallmatrix}\right) \in \slz \eqqcolon \Gamma$. The group $\Gamma$ acts on $\H \cup \R \cup \{i\infty\}$ by Möbius transformations. We stipulate $\tau = u+iv\in \H$ throughout, and write $j(\gamma, \tau) \coloneqq (c\tau+d)$ for the usual modular multiplier. 

\subsubsection{Classification of motions}
We summarize some standard facts.
\begin{enumerate}[label=(\alph*)]
\item An element $\gamma \in \Gamma\setminus\{\pm\Id\}$ is called parabolic if $\vt{\tr{(\gamma)}} = 2$. We have a unique fixed point $\mathfrak{a}_{\gamma}$ of $\gamma$, called a cusp, and located in $\Q \cup \{i\infty\}$. The stabilizer of each cusp is conjugate to the stabilizer of $i\infty$, which is generated by $T \coloneqq \left(\begin{smallmatrix} 1 & 1 \\ 0 & 1 \end{smallmatrix}\right)$ up to sign. In other words $\gamma = \pm \sigma_{\mathfrak{a}_{\gamma}} T^n \sigma_{\mathfrak{a}_{\gamma}}^{-1}$ for some $n \in \Z\setminus\{0\}$, where $\sigma_{\mathfrak{a}_{\gamma}}$ is a scaling matrix of the cusp, namely it satisfies $\sigma_{\mathfrak{a}_{\gamma}} \infty = \mathfrak{a}_{\gamma}$. Points are moved by $\gamma$ along horocycles, that are circles in $\H$ tangent to $\R$.
\item An element $\gamma \in \Gamma$ is called elliptic if $\vt{\tr{(\gamma)}} < 2$. Recall that any elliptic fixed point $w_{\gamma}$ is $\Gamma$-equivalent to either $i$ or $\omega \coloneqq \mathrm{e}^{\frac{\pi i}{3}}$. Letting $S \coloneqq \left(\begin{smallmatrix} 0 & -1 \\ 1 & 0 \end{smallmatrix}\right)$, $U \coloneqq TS = \left(\begin{smallmatrix} 1 & -1 \\ 1 & 0 \end{smallmatrix}\right)$, we see that $\Gamma_i = \left\{\Id, S, S^2, S^3\right\}$, $\Gamma_{\omega} = \left\{\Id, U, \ldots, U^5\right\}$. Points are moved by $\gamma$ along circles centered at $w_\gamma$.
\item An element $\gamma \in \Gamma$ is called hyperbolic if $\vt{\tr{(\gamma)}} > 2$. Recall that $\gamma$ has precisely two different fixed points $w_{\gamma}$, $w'_{\gamma}$, located on the real axis. Writing $\Gamma_{w_\gamma, w'_{\gamma}} = \pm \langle \eta_{w_\gamma, w'_{\gamma}} \rangle$, there exists a scaling matrix $\sigma_{w_\gamma, w'_{\gamma}} \in \mathrm{SL}_2(\R)$ such that $\sigma_{w_\gamma, w'_{\gamma}}0 = w_{\gamma}$, $\sigma_{w_\gamma, w'_{\gamma}}\infty = w'_{\gamma}$, and $\sigma_{w_\gamma, w'_{\gamma}}^{-1} \eta_{w_\gamma, w'_{\gamma}} \sigma_{w_\gamma, w'_{\gamma}} = \pm \left(\begin{smallmatrix} y & 0 \\ 0 & y^{-1} \end{smallmatrix}\right)$ for some $y \in \R_{>0}$. Points are moved by $\gamma$ along hypercycles, that are lines and circle arcs intersecting $\R$ at non-perpendicular angles.
\end{enumerate}

\subsection{Integral binary quadratic forms}
Let $Q$ be an integral binary quadratic form, and the terminology ``quadratic form'' abbreviates such forms throughout. The group $\Gamma$ acts on the set of quadratic forms by defining $\left(Q \circ \left(\begin{smallmatrix} a & b \\ c & d \end{smallmatrix}\right)\right)(x,y)$ as $Q(ax+by, cx+dy)$, and this induces an equivalence relation, which we denote by $\sim$. Moreover, the actions of $\Gamma$ on $\H$ and on quadratic forms are compatible, in the sense that $\left(Q \circ \gamma\right)(\tau,1)$ equals $j(\gamma,\tau)^2 Q(\gamma\tau,1)$. Sometimes, we abbreviate $\left[a,b,c\right] \coloneqq ax^2+bxy+cy^2$, and we denote its discriminant $b^2-4ac$ by $\Delta([a,b,c])$. One can check that the discriminant is invariant under $\sim$. For every $D \in \Z$, we let $\Qc(D) \coloneqq \left\{Q \colon \Delta(Q) = D\right\}$ be the set of all quadratic forms with discriminant $D$. If $D \neq 0$ the set $\Qc(D)_{\sim} \coloneqq \Qc(D) \slash \Gamma$ is finite, whose cardinality is called the class number $h(D)$. If $D \equiv 0 \pmod*{4}$ or $D \equiv 1 \pmod*{4}$, then $\Qc(D)_{\sim}$ is non-empty.

\subsubsection{Heegner geodesics}
Let $0 \neq Q$ be a quadratic form. If $\Delta(Q) > 0$ then we associate to $Q$ the Heegner geodesic $S_Q \coloneqq \{\tau \in \H \ \colon a\vt{\tau}^2+b\Re{(\tau)}+c=0 \}$, that is $S_Q$ is an arc in $\H$ perpendicular to the real axis joining the two distinct zeros of $Q(\tau,1)$. If $a=0$ then the second point is given by $-\frac{c}{b}$.

\subsubsection{Quadratic forms associated to $\gamma \in \Gamma$}
In addition, we define $Q_{\gamma}(x,y)$ to be the quadratic form $cx^2+(d-a)xy-by^2$ associated to $\gamma = \left(\begin{smallmatrix} a & b \\ c & d \end{smallmatrix}\right) \in \Gamma$. We set $\Delta(\gamma) \coloneqq \Delta\left(Q_{\gamma}\right) = \tr{(\gamma)}^2 - 4$, and observe that the sign of $\Delta(\gamma)$ depends precisely on hyperbolicity, parabolicity, or ellipticity of $\gamma$ respectively. Futhermore, we note that $Q_{-\gamma}(x,y) = Q_{\gamma^{-1}}(x,y) = -Q_{\gamma}(x,y)$. Hence, we invoke a sign-function on quadratic forms. Namely, we define 
\begin{align*}
\sgn{\left([a,b,c]\right)} \coloneqq \begin{cases}
\sgn(a) & \text{ if } a \neq 0, \\
\sgn(c) & \text{ if } a = 0.
\end{cases}
\end{align*}
This will cause a difference in the case of positive discriminant only.
\begin{lemma}
Suppose $\Delta(Q) \leq 0$. Then $Q \sim -Q$ implies $Q = 0$.
\end{lemma}

\subsubsection{Genus characters}
This subsection follows the introduction given by Gross, Kohnen, Zagier in \cite[p.\ 508]{grokoza}. Let $Q = [a,b,c]$ be a quadratic form. We observe that $\sim$ preserves $\gcd{(a,b,c)}$ as well. We would like to define a $\Gamma$-invariant function on $\Qc(D)$ (assume $D \equiv 0 \pmod*{4}$ or $D \equiv 1 \pmod*{4}$). If $D \neq 0$, let $d$ be a fundamental discriminant dividing $D$, and let $\big(\frac{d}{\cdot}\big)$ be the Kronecker symbol. In addition, an integer $n$ is represented by $Q$ if there exist $x$, $y \in \Z$, such that $Q(x,y) = n$. This established, we define
\begin{align*}
\chi_d\left([a,b,c]\right) &\coloneqq \begin{cases}
\left(\frac{d}{n}\right) & \text{ if } \gcd{(a,b,c,d)} = 1, [a,b,c] \text{ represents } n, \gcd{(d,n)} = 1, \\
0 & \text{ if } \gcd{(a,b,c,d)} > 1.
\end{cases}
\end{align*}
One can verify that such an integer $n$ always exists, and the definition is independent from its choice. Since equivalent quadratic forms represent the same integers, this function is indeed invariant under $\sim$. The choice $d=1$ yields the trivial character. The definition of $\chi_d\left([a,b,c]\right)$ extends to $D=0$ by choosing $d=0$ in this case, compare the proof of Lemma \ref{lem:twistedconv}. Additional properties of $\chi_d$ are summarized in \cite[Proposition 1 and 2]{grokoza}.

\section{Construction of Eisenstein series}

\subsection{Eisenstein series associated to a quadratic form}

This construction is based on the following two observations, and follows \cite{matsu2}.
\begin{lemma}
Let $\gamma \in \Gamma\setminus\{\pm\Id\}$, and $Q_{\gamma}$ be the associated quadratic form to $\gamma$.
\begin{enumerate}[label=(\roman*)]
\item The zeros of $Q_{\gamma}(\tau,1)$ are precisely the fixed points of $\gamma$ in $\H \cup \R$.
\item The equivalence class of $Q_{\gamma}(\tau,1)$ is precisely the set $\left\{Q_{\alpha^{-1}\gamma \alpha}(\tau,1) \colon \alpha \in \Gamma\right\}$.
\end{enumerate}
\end{lemma}

We observe that division by $Q_{\gamma}(\cdot,1)$ and averaging over equivalence classes of $Q_{\gamma}(\cdot,1)$ modulo its zeros provides a function of weight $2$. Consequently, we define the following functions.
\begin{defn}
Let $\gamma \in \Gamma\setminus\{\pm\Id\}$, $k \in 2\N$, $\tau \in \H$, $\Re{(s)} > 1 -\frac{k}{2}$, and $w(\gamma)$ be the set of fixed points of $\gamma$. Then we define
\begin{align*}
E_{k, Q}(\tau,s) &\coloneqq \sum_{Q' \sim Q} \frac{\sgn{(Q')}^{\frac{k}{2}} v^s}{Q'(\tau,1)^{\frac{k}{2}} \vt{Q'(\tau,1)}^s}, \\
E_{k, \gamma}(\tau,s) &\coloneqq E_{k, Q_{\gamma}}(\tau,s) = \sum_{\alpha \in \bsfrac{\Gamma}{\Gamma_{w(\gamma)}}} \frac{\sgn{(Q_{\alpha^{-1}\gamma \alpha})}^{\frac{k}{2}} v^s}{Q_{\alpha^{-1}\gamma \alpha}(\tau,1)^{\frac{k}{2}} \vt{Q_{\alpha^{-1}\gamma \alpha}(\tau,1)}^s}.
\end{align*}
\end{defn}
We establish convergence.
\begin{lemma}
For every $k \in 2\N$ the series defining $E_{k, Q}(\tau,s)$ converges absolutely and locally uniformly for $\tau \in \H$ and $\Re{(s)} > 1-\frac{k}{2}$.
\end{lemma}

\begin{proof}
This follows by results of Petersson \cite[Satz 1, Satz 4, Satz 6]{pet48}.
\end{proof}

However, $E_{k, \gamma}$ is not modular yet, because the sign-function is not invariant under equivalence of quadratic forms. The circumvention of this obstruction depends on the motion $\gamma$ induces.

\subsection{Eisenstein series associated to a given discriminant}
Let $D \equiv 0 \pmod*{4}$ or $D \equiv 1 \pmod*{4}$. If $D \neq 0$, we let $d$ be the positive fundamental discriminant dividing $D$, else we set $d=0$. We average over $\Qc(D)$. Henceforth, we twist the average by a genus character, and split the sum into equivalence classes (recall that such a character descends to $\Qc(D)_{\sim}$).
\begin{defn}
Let $\gamma \in \Gamma\setminus\{\pm\Id\}$, $k \in 2\N$, $\tau \in \H$, $\Re{(s)} > 1 -\frac{k}{2}$. Then we define
\begin{align*}
\Ec_{k,D}(\tau, s) &\coloneqq \sum_{0 \neq Q \in \Qc(D)_{\sim}} \chi_d\left(Q\right) E_{k,Q}(\tau,s), \\
\Ec_{k,\gamma} (\tau,s) &\coloneqq \Ec_{k,\Delta(\gamma)}(\tau, s) = \sum_{0 \neq Q \in \Qc\left(\Delta(\gamma)\right)_{\sim}} \chi_d\left(Q\right) E_{k,Q}(\tau,s)
\end{align*}
\end{defn}

We establish convergence.
\begin{lemma} \label{lem:twistedconv}
For every $k \in 2\N$, $\tau \in \H$, and $\Re{(s)} > 1-\frac{k}{2}$ the series defining $\Ec_{k, D}(\tau,s)$ converges absolutely and locally uniformly.
\end{lemma}

\begin{proof}
If $D = 0$, then $\chi_0(Q) = 0$ except $Q$ is primitive, and represents $\pm 1$. Thus, we reduce to the quadratic forms $\left[c^2,2cd,d^2\right]$ for any coprime pair $(c,d) \in \Z^2$. But such a quadratic form is equivalent to either $[-1,0,0]$ or $[1,0,0]$. If $D \neq 0$ the class number $h(D)$ is finite. This proves the claim.
\end{proof}

\section{Parabolic and elliptic Eisenstein series} \label{sec:peeis}

\subsection{Parabolic case}
It sufficed to study the case $\gamma = T^n$. Let $\alpha = \left(\begin{smallmatrix} * & * \\ c & d \end{smallmatrix}\right) \in \bsfrac{\Gamma}{\Gamma_{\infty}}$. We compute
\begin{align*}
\alpha^{-1}T^n \alpha = \left(\begin{array}{ccc} 1+cdn & d^2n \\ -c^2n & 1-cdn \end{array}\right), \qquad Q_{\alpha^{-1}T^n\alpha}(\tau,1) = -nj(\gamma,\tau)^2.
\end{align*}
Hence, using $\sgn(n) = \frac{n}{\vt{n}}$, we recover the usual real analytic Eisenstein series
\begin{align*}
E_{k, \gamma}(\tau,s) = E_{k, T^n}(\tau,s) = \frac{1}{\vt{n}^{s+\frac{k}{2}}} \left( \sum_{\alpha \in \bsfrac{\Gamma}{\Gamma_{\infty}}} \frac{\Im{(\alpha\tau)}^s}{j(\alpha,\tau)^k}\right).
\end{align*}
We infer
\begin{align*}
\Ec_{k,\gamma}(\tau,s) = 2\sum_{\alpha \in \bsfrac{\Gamma}{\Gamma_{\infty}}} \frac{\Im{(\alpha\tau)}^s}{j(\alpha,\tau)^k}
\end{align*}
for any parabolic motion $\gamma$.

\subsubsection{Modularity}
Clearly $\Ec_{k,\gamma}$ is modular of weight $k$ according to the ``chain rule'' property of the modular multiplier, that is $j(\alpha\beta,\tau) = j(\alpha,\beta\tau)j(\beta,\tau)$ for every $\alpha,\beta \in \Gamma$ and every $\tau \in \H$.

\subsubsection{Analytic continuation}
It is a classical fact that $\Ec_{k,\gamma}$ can be continued meromorphically to the whole $s$-plane, see \cite[p.\ 76-79]{sel56}, \cite[p.\ 293]{roe67}. Be aware of the fact that Roelcke uses the automorphy factor $\big(\frac{j(\gamma,\tau)}{\vt{j(\gamma, \tau)}}\big)^{-k}$,
whence his initial domain of convergence is $\Re{(s)} > 1$ for every $k \in \R$. 

If $k > 2$ we may simply insert $s=0$, and obtain the classical holomorphic modular Eisenstein series
\begin{align*}
\Ec_{k,\gamma}(\tau,0) = \sum_{\mathrm{gcd}(c,d)=1} \frac{1}{\left(c\tau+d\right)^k}.
\end{align*}

If $k=2$ we utilize Hecke's trick (cf.\ Zagier \cite[p.\ 19-20]{the123}), or alternatively the Fourier expansion of $E_{2, T^{\pm 1}}(\tau,s)$ (cf.\ Iwaniec \cite[p.\ 51]{iwaniec97}), to achieve
\begin{align*}
\lim_{s \searrow 0} \Ec_{k,\gamma}(\tau,s) = 2 E_2^*(\tau) \coloneqq 2\left(E_2(\tau) - \frac{3}{\pi v}\right)\coloneqq 2\left(1 - 24\sum_{n \geq 1} \sum_{d \mid n} d \ q^n - \frac{3}{\pi v}\right).
\end{align*}
This is the holomorphic Eisenstein series of weight $2$ completed by $-\frac{3}{\pi v}$, and thus a harmonic Maa{\ss} form of weight $2$. A good exposition on the theory as well as on the applications of harmonic Maa{\ss} forms can be found in \cite{thebook}.

\subsection{Elliptic case}
Recall that any elliptic motion is conjugate to either $S$ or $U$, so it suffices to deal with those two cases, up to a change of sign and class numbers. Those cases correspond to discriminants $-4$ and $-3$ respectively, and both class numbers are equal to $1$. Reduced primitive representatives are $[1,0,1]$, $[1,1,1]$, and the genus character of both forms equals $1$. Hence, it suffices to investigate\footnote{Note that the two cases do not cover the more general case of $E_{k,Q}(\tau,s)$ with $\Delta(Q) < 0$. However, one may reuse the function $E_k(\tau,z,s)$ to deal with this case.} $E_{k,S}$ and $E_{k,U}$. To this end, we define the following function.
\begin{defn}
Let $\tau, z \in \H$ be $\Gamma$-inequivalent to each other, and $\Re{(s)} > 1-\frac{k}{2}$, $k \in 2\N$. Then we define
\begin{align*}
E_k\left(\tau, z, s\right) \coloneqq \sum_{\alpha \in \Gamma} \frac{\Im{(z)^{s+\frac{k}{2}}\Im{(\alpha\tau)}^s}}{j(\alpha,\tau)^k \left(\alpha\tau-z\right)^{\frac{k}{2}}\left(\alpha\tau-\overline{z}\right)^{\frac{k}{2}}\vt{\left(\alpha\tau-z\right)\left(\alpha\tau-\overline{z}\right)}^s}.
\end{align*}
\end{defn}

\subsubsection{Modularity}
The function $E_k$ enjoys the following properties.
\begin{lemma}
\begin{enumerate}[label=(\roman*)]
\item If $\gamma = S$ or $\gamma = U$, then $\sgn{(Q_{\alpha^{-1}\gamma \alpha})} = 1$ for any $\alpha \in \Gamma$, and 
\begin{align*}
E_{k,\gamma}(\tau,s) = \frac{\Im{(w_{\gamma})}^{-s-\frac{k}{2}}}{\vt{\Gamma_{w_\gamma}}} E_k\left(\tau,w_\gamma,s\right).
\end{align*}
in both cases.
\item For any $\rho \in \Gamma$ we have
\begin{align*}
E_k\left(\rho\tau,z,s\right) = j(\rho,\tau)^k E_k\left(\tau,z,s\right), \quad E_k\left(\tau,\rho z,s\right) = E_k\left(\tau,z,s\right).
\end{align*}
\end{enumerate}
\end{lemma}

\begin{proof}
Both items can be checked by computation, and we provide the main steps.
\begin{enumerate}[label=(\roman*)]
\item Suppose $\gamma = S,U$ and $w_{\gamma} = i,\omega$. Letting $\alpha = \left(\begin{smallmatrix} a & b \\ c & d \end{smallmatrix}\right) \in \Gamma $, we compute
\begin{align*}
Q_{\alpha^{-1}S\alpha}(\tau,1) &= \left(a^2+c^2\right)\tau^2 + 2\left(ab+cd\right)\tau + b^2+d^2, \\
Q_{\alpha^{-1}U\alpha}(\tau,1) &= \left(a^2+c^2-ac\right)\tau^2 + \left(2ab+2cd-ad-bc\right)\tau + b^2+d^2-bd,
\end{align*}
which implies the first claim, and additionally
\begin{align*}
j(\alpha,\tau)^2 \left(\alpha\tau-w_{\gamma}\right)\left(\alpha\tau-\overline{w_\gamma}\right) = Q_{\alpha^{-1}\gamma \alpha}(\tau,1)
\end{align*}
in both cases. The second claim follows directly.
\item One checks the following two identities. For every $\alpha,\beta \in \Gamma$ and every $\tau,z \in \H$ we have
\begin{align*}
j(\alpha,\tau) j(\alpha,\overline{\tau}) = \vt{j(\alpha,\tau)}^2, \qquad j(\beta,z)(\tau-\beta z) = j\left(\beta^{-1},\tau\right)\left(\beta^{-1}\tau-z\right).
\end{align*}
Modularity in $\tau$ follows directly by the ``chain rule'' property of the modular multiplier. To show modularity in $z$ we substitute $\alpha = \rho\beta$, and see that
\begin{align*}
\left(\vt{j(\rho,z)}^{2s+k} \vt{j(\rho, \beta\tau)}^{2s} j(\rho \beta,\tau)^{k} \frac{j(\rho^{-1},\rho\beta\tau)^k}{j(\rho, z)^{\frac{k}{2}}j(\rho, \overline{z})^{\frac{k}{2}}} \vt{\frac{j(\rho^{-1},\rho\beta\tau)^2}{j(\rho, z) j(\rho, \overline{z})}}^s\right)^{-1}
\end{align*}
equals $j(\beta,\tau)^{-k}$. This proves the second item. \qedhere
\end{enumerate}
\end{proof}

\subsubsection{Analytic continuation in weight $2$}
To describe the analytic continuation we recall Petersson's Poincar\'{e} series.
\begin{defn}
Let $\tau, z \in \H$ be $\Gamma$-inequivalent to each other, and $\Re{(s)} > 1-\frac{k}{2}$, $k \in 2\N$. Then we define
\begin{align*}
P_k\left(\tau, z, s\right) \coloneqq \sum_{\alpha \in \Gamma} \frac{\Im{(z)}^{s+\frac{k}{2}}}{j(\alpha,\tau)^k \vt{j(\alpha,\tau)}^{2s} \left(\alpha\tau-z\right)^{\frac{k}{2}} \left(\alpha\tau-\overline{z}\right)^{\frac{k}{2}} \vt{\alpha\tau-\overline{z}}^{2s}}.
\end{align*}
\end{defn}

The series $P_k$ enjoys the following transformation properties.
\begin{lemma}
Let $\rho \in \Gamma$. Then
\begin{align*}
\Im{(\rho\tau)}^s P_k\left(\rho\tau, z, s\right) = j(\rho,\tau)^k v^s P_k\left(\tau, z, s\right), \quad P_k\left(\tau, \rho z, s\right) = P_k\left(\tau, z, s\right).
\end{align*}
\end{lemma}

\begin{proof}
This follows by the same argument as in the case of $E_k\left(\tau, z, s\right)$.
\end{proof}

The analytic continuation of $P_2$ to $s=0$ was established by Petersson \cite{pet44}, and to $\Re{(s)} > -\frac{1}{4}$ by Bringmann, Kane \cite[Theorem 3.1]{brika}. To describe it, let $j(\tau)$ be Klein's modular invariant for $\Gamma$, and $D \coloneqq \frac{1}{2\pi i}\frac{d}{d\tau}$. Then Asai, Kaneko, Ninomiya discovered in \cite[Theorem 3]{askani} the Fourier expansion
\begin{align*}
\frac{D(j)(\tau)}{j(w)-j(\tau)} = \sum_{m \geq 0} j_m(w) q^m, \qquad \Im{(\tau)} > \Im{(w)}
\end{align*}
where $j_m(w)$ is the unique element in $\C\left[j(w)\right]$ of the shape
\begin{align*}
j_m(w) = \mathrm{e}^{-2 \pi i m w} + O\left(\mathrm{e}^{2 \pi i w}\right).
\end{align*}
In \cite{brikaloeonro}, the authors proved that the functions $j_m(w)$ form a Hecke system, namely if $T_m$ denotes the normalized Hecke operator, then
\begin{align*}
j_0(w) = 1, \qquad j_1(w) = j(w)-744, \qquad j_m(w) = \left(T_mj_1\right)(w).
\end{align*}
Afterwards, they simplified the expressions from \cite[Theorem 3.1]{brika}, based on earlier work of Duke, Imamo\={g}lu, T\'{o}th \cite[Theorem 5]{duimto11}, and proved that
\begin{align*}
\lim_{s \searrow 0} P_{2}(\tau,z,s) = -2\pi\left(\frac{D(j)(\tau)}{j(z)-j(\tau)}-E_2^*(\tau)\right).
\end{align*}
Matsusaka \cite[Theorem 2.3]{matsu2} extended the latter result to $E_2(\tau,z,s)$, especially
\begin{align*}
\lim_{s \searrow 0} E_{2}(\tau,z,s) = -2\pi\left(\frac{D(j)(\tau)}{j(z)-j(\tau)}-E_2^*(\tau)\right),
\end{align*}
which in turn provides the analytic continuation of $E_{2,\gamma}(\tau,s)$ in the elliptic case. 

\begin{rmk}
Note that the analytic continuation is a polar harmonic Maa{\ss} form of weight $2$ on $\Gamma$ in $\tau$. The poles are located on $\Gamma z$. Such forms satisfy all conditions of an ordinary harmonic Maa{\ss} form, but are permitted to have poles in $\H$. See \cite[Section 13.3]{thebook} for more details.
\end{rmk}

\section{Hyperbolic Eisenstein series} \label{sec:hypeis}
Let $\gamma \in \Gamma$ be hyperbolic. Thus $\Delta(\gamma) > 0$, and $\Delta(\gamma)$ is not a square number. The two fixed points $w_{\gamma}$, $w'_{\gamma}$ of $\gamma$ are real quadratic irrationals, which are Galois conjugate to each other. The geodesic $S_{Q_{\gamma}}$ is an arc in $\H$ connecting $w_{\gamma}$ and $w'_{\gamma}$ (equivalently, the two zeros of $Q_{\gamma}(\tau,1)$), which is perpendicular to $\R$. 

\subsection{Fourier expansion in general}
We suppose in addition that $\gamma$ is primitive throughout, that is the stabilizer $\Gamma_{w_\gamma, w'_{\gamma}}$ is infinite cyclic, and generated by $\gamma$.

\subsubsection{Fourier expansion of $E_{k,\gamma}$}
We appeal to Zagier's method \cite[Section 2]{zagier75}. We recall the double coset decomposition $\Gamma_{w_\gamma, w'_{\gamma}}\backslash\Gamma \slash \Gamma_{\infty}$, and unfold
\begin{align*}
E_{k,\gamma}(\tau, s) = \sum_{\alpha \in \Gamma_{w_\gamma, w'_{\gamma}}\backslash\Gamma \slash \Gamma_{\infty}} \sum_{\beta \in \langle T \rangle} \frac{\sgn{(Q_{(\alpha\beta)^{-1}\gamma (\alpha\beta)})}^{\frac{k}{2}} \Im{(\alpha\beta\tau)}^s} {j(\alpha\beta,\tau)^k Q_{\gamma}(\alpha\beta\tau,1)^{\frac{k}{2}} \vt{Q_{\gamma}(\alpha\beta\tau,1)}^s}.
\end{align*}

We observe that the innermost sum is one-periodic, and hence has a Fourier expansion
\begin{align*}
E_{k,\gamma}(u+iv, s) = \sum_{a \in \Z\setminus\{0\} } \sum_{\substack{b \pmod*{2a} \\ \big[a,b,\frac{b^2-\Delta(\gamma)}{4a}\big] \sim Q_{\gamma}}} \sum_{m \in \Z} c_m(v,s) \mathrm{e}^{2 \pi i m u}
\end{align*}
with coefficients $c_m(v,s)$ given by
\begin{align*}
\int_{-\infty}^{\infty} \frac{v^s \sgn{(a)}^{\frac{k}{2}}\mathrm{e}^{-2\pi i m x}}{\left(a(x+iv)^2+b(x+iv)+\frac{b^2-\Delta(\gamma)}{4a}\right)^{\frac{k}{2}} \vt{a(x+iv)^2+b(x+iv)+\frac{b^2-\Delta(\gamma)}{4a}}^s} dx.
\end{align*}

We abbreviate
\begin{align*}
\lambda \coloneqq \frac{\sqrt{\Delta(\gamma)}}{2\vt{a}} > 0,
\end{align*}
and substitute $x+iv \eqqcolon it - \frac{b}{2a}$. We infer
\begin{align*}
c_m(v,s) &= \frac{-i v^s \mathrm{e}^{2\pi i m \left(\frac{b}{2a}+iv\right)}}{(-1)^{\frac{k}{2}}\vt{a}^{\frac{k}{2}+s}} \int_{v-i\infty}^{v+i\infty} \frac{\mathrm{e}^{2\pi m t}}{\left(t^2+\lambda^2\right)^{\frac{k}{2}} \vt{t^2+\lambda^2}^s} dt.
\end{align*}

Splitting the integral at $v \pm i\lambda$ and majorizing each, we obtain the following result.
\begin{lemma} \label{lem:iconv}
If $\Re{(s)} > 1-\frac{k}{2}$, then the integral expression defining $c_m(v,s)$ converges absolutely.
\end{lemma}

\subsubsection{Fourier expansion of $\Ec_{k,\gamma}$}
Now, we turn our interest to the Fourier expansion of $\Ec_{k,\gamma}$. We let $d$ be the positive fundamental discriminant dividing $\Delta(\gamma)$, and in addition, we let
\begin{align*}
W_Q(m,a) \coloneqq \sum_{\substack{b \pmod*{2a} \\ \big[a,b,\frac{b^2-\Delta(\gamma)}{4a}\big] \sim Q}} \mathrm{e}^{\pi i m \frac{b}{a}}, \qquad a \in \Z \setminus\{0\}, \quad m \in \Z
\end{align*}
which is a so called quadratic Weyl sum. Now, the additional averaging over $\Qc_{\sim}$ comes in handy.
\begin{lemma} \label{lem:qsymm}
Let $m \in \Z$. Then we have
\begin{align*}
2\sum_{Q \in \Qc\left(\Delta(\gamma)\right)_{\sim}} \chi_d(Q) \sum_{a \geq 1} W_Q(m,a)
&= \sum_{Q \in \Qc\left(\Delta(\gamma)\right)_{\sim}} \chi_d(Q) \sum_{a \in \Z\setminus\{0\}}W_Q(m,a) \\
&= \sum_{Q \in \Qc\left(\Delta(\gamma)\right)_{\sim}} \chi_d(Q) \sum_{a \in \Z\setminus\{0\}} W_Q(-m,a).
\end{align*}
\end{lemma}

\begin{proof}
We prove the first equality, and observe that $a \mapsto -a$ yields
\begin{align*}
\sum_{\substack{b \pmod*{-2a} \\ \big[-a,b,\frac{b^2-\Delta(\gamma)}{-4a}\big] \sim Q}} \mathrm{e}^{\pi i m \frac{b}{-a}} \ = \ \sum_{\substack{b \pmod*{2a} \\ \big[-a,b,\frac{b^2-\Delta(\gamma)}{-4a}\big] \sim Q}} \mathrm{e}^{\pi i m \frac{-b}{a}} \ = \ \sum_{\substack{b \pmod*{2a} \\ \big[-a,-b,\frac{b^2-\Delta(\gamma)}{-4a}\big] \sim Q}} \mathrm{e}^{\pi i m \frac{b}{a}}
\end{align*}
by reordering summands. We compute
\begin{align*}
\Delta\left(\left[-a,-b,\frac{b^2-\Delta(\gamma)}{-4a}\right]\right) = (-b)^2 - 4(-a)\frac{b^2-\Delta(\gamma)}{-4a} = \Delta(\gamma).
\end{align*}
Furthermore, we have $\chi_d(-Q) = \sgn{(d)}\chi_d(Q)$. Indeed, suppose $Q$ represents some $n$, and $\gcd{(d,n)} = 1$. Then $-Q$ represents $-n$, and $\gcd{(d,-n)} = 1$. This enables us to write $\left(\frac{d}{-n}\right) = \left(\frac{d}{-1}\right)\left(\frac{d}{n}\right)$, and $\left(\frac{d}{-1}\right) = \sgn(d) = 1$. In other words, changing the sign of $Q$ permutes quadratic forms of discriminant $\Delta(\gamma)$ up to equivalence. The second equality follows analogously.
\end{proof}

We deduce that $\Ec_{k,\gamma}(\tau, s)$ equals ($\lambda = \frac{\sqrt{\Delta(\gamma)}}{2a}$)
\begin{align*}
\frac{-2i v^s}{(-1)^{\frac{k}{2}}} \sum_{m \in \Z}  \sum_{Q \in \Qc\left(\Delta(\gamma)\right)_{\sim}} \chi_d(Q) \sum_{a \geq 1 } \frac{W_Q(m,a)}{a^{\frac{k}{2}+s}} \int_{v-i\infty}^{v+i\infty} \frac{\mathrm{e}^{2\pi m t}}{\left(t^2+\lambda^2\right)^{\frac{k}{2}} \vt{t^2+\lambda^2}^s} dt \ q^m.
\end{align*}

We re-establish convergence.
\begin{lemma} \label{lem:fconv}
Suppose $\Re(s) > 1-\frac{k}{2}$. Then the Fourier expansion of $\Ec_{k,\gamma}(\tau, s)$ converges absolutely.
\end{lemma}

Before the proof, we rewrite the Fourier expansion of $\Ec_{k,\gamma}$. To this end, we define $d'$ by $0 < \Delta(\gamma) = dd'$ with $d$ fundamental, and recall the Sali{\'e} sum
\begin{align*}
T_m(d,d',c) \coloneqq \sum_{\substack{b \pmod*{c} \\ b^2 \equiv dd' \pmod*{c}}} \chi_d\left(\left[\frac{c}{4},b,\frac{b^2-dd'}{c}\right]\right) \mathrm{e}^{2\pi i \left(\frac{2mb}{c}\right)}.
\end{align*}
We observe that $T_m(d,d',4a)$ equals
\begin{align*}
2 \sum_{\substack{b \pmod*{2a} \\ b^2 \equiv dd' \pmod*{4a}}} \chi_d\left(\left[a,b,\frac{b^2-dd'}{4a}\right]\right) \mathrm{e}^{\pi i \left(\frac{mb}{a}\right)} = 2\sum_{Q \in \Qc\left(\Delta(\gamma)\right)_{\sim}} \chi_d(Q) W_Q(m,a).
\end{align*}
Thus, we arrive at
\begin{align*}
\Ec_{k,\gamma}(\tau, s) = \frac{-i v^s}{(-1)^{\frac{k}{2}}} \sum_{m \in \Z}  \sum_{a \geq 1 } \frac{T_m(d,d',4a)}{a^{\frac{k}{2}+s}} \int_{v-i\infty}^{v+i\infty} \frac{\mathrm{e}^{2\pi m t}}{\left(t^2+\lambda^2\right)^{\frac{k}{2}} \vt{t^2+\lambda^2}^s} dt \ q^m.
\end{align*}

\begin{proof}[Proof of Lemma \ref{lem:fconv}]
We utilize the Weil bound
\begin{align*}
T_m(d,d',4a) \ll \gcd\left(d',m^2d,4a\right)^{\frac{1}{2}} (4a)^{\varepsilon}
\end{align*}
for any $\varepsilon > 0$, compare \cite[p.\ 1545]{andu} for instance. Hence, we have
\begin{align*}
\vt{\sum_{a \geq 1 } \frac{T_m(d,d',4a)}{a^{\frac{k}{2}+s}}} \ll \sum_{a \geq 1 } \frac{a^{\varepsilon}}{a^{\frac{k}{2}+\Re(s)}}
\end{align*}
by bounding $\gcd\left(d',m^2d,4a\right) \ll 1$ (see \cite[p.\ 965]{duimto11} as well), which converges absolutely for any $s \in \C$ with $\Re(s) > 1-\frac{k}{2}$. We conclude by Lemma \ref{lem:iconv}.
\end{proof}

\subsection{Evaluation at $s=0$}
Let
\begin{align*}
J_{\mu}(x) \coloneqq \left(\frac{x}{2}\right)^{\mu} \sum_{j \geq 0} \frac{(-1)^j}{j! \ \Gamma(\mu + j + 1)}\left(\frac{x}{2}\right)^{2j}
\end{align*}
be the usual $J$-Bessel function. We compute the inverse Laplace transform from above.
\begin{lemma} \label{lem:laplace}
Let $\rho > 0$, $m \in \Z$. Then
\begin{align*}
\frac{1}{2 \pi i}\int_{v-i\infty}^{v+i\infty} \frac{\mathrm{e}^{2\pi m t}}{\left(t^2+\lambda^2\right)^{\rho}} dt = \begin{cases}
\frac{\sqrt{\pi}}{\Gamma(\rho)} \left(\frac{\pi m}{\lambda}\right)^{\rho-\frac{1}{2}} J_{\rho-\frac{1}{2}}(2 \pi \lambda m) & \text{ if } m > 0, \\
0 & \text{ if } m \leq 0.
\end{cases}
\end{align*}
\end{lemma}

\begin{proof}
The case $m > 0$ follows directly by item \cite[eq.\ 29.3.57]{abraste}. If $m \leq 0$, we see that the poles of the integrand are on the imaginary axis, to the left of the contour of integration. Hence, we may deform the contour to the right up to $i\infty$ without including any poles, see \cite[p.\ 249]{koh85}. Since $m \leq 0$, the integrand is holomorphic at $i\infty$ as well, and the claim follows by Cauchy's theorem.
\end{proof}

Next, we invoke the $I$-Bessel function $I_{\mu}(x) \coloneqq i^{-\mu} J_{\mu}(ix)$ to define the auxiliary function
\begin{align*}
\phi_m(y,s) \coloneqq \begin{cases}
y^s & \text{ if } m = 0, \\
2\pi \sqrt{\vt{m}y} \ I_{s-\frac{1}{2}}\left(2\pi \vt{m} y\right) & \text{ if } m \in \Z\setminus\{0\}.
\end{cases}
\end{align*}
Averaging this function gives rise to the Niebur Poincar\'{e} series \cite{neu73, nie73}
\begin{align*}
G_m(\tau,s) \coloneqq \sum_{\alpha \in \bsfrac{\Gamma}{\Gamma_{\infty}}} \phi_m\left(\Im{(\alpha \tau)},s\right) \mathrm{e}^{2\pi i m \Re{(\alpha \tau)}}, \quad \Re(s) > 1.
\end{align*}
The analytic properties of $G_m$ can be easily derived from its Fourier expansion, see \cite[pp.\ 969-970]{duimto11} for instance. In particular, $G_m(\cdot,s)$ is invariant under the action of $\Gamma$ on $\H$. 
Moreover, recall the notation $\Gamma_Q$ for the stabilizer of the two zeros of $Q$. The main ingredient is the following result due to Duke, Imamo\={g}lu, T\'{o}th.

\begin{lemma}[\protect{\cite[Proposititon 4]{duimto11}}\footnote{See also \cite[Lemma 4]{duimto16}, \cite[eq.\ (2-7)]{andu}}.] \label{lem:duimto}
Let $\Re(\rho) > 1$, $m \in \Z$, $\Delta(\gamma) = dd' > 0$ with $d > 0$ fundamental, and $d \neq d'$. Then, we have
\begin{align*}
&\frac{\Gamma(\rho)}{2^\rho\Gamma\left(\frac{\rho}{2}\right)^2}\sum_{Q \in \Qc\left(\Delta(\gamma)\right)_{\sim}} \chi_d(Q) \int_{\bsfrac{S_Q}{\Gamma_Q}} G_m(w,\rho) \frac{\vt{dw}}{\Im{(w)}} \\
&= \begin{cases}
\sqrt{2}\pi\vt{m}^{\frac{1}{2}}\Delta(\gamma)^{\frac{1}{4}} \sum_{0 < c \equiv 0 \pmod*{4}} \frac{T_m(d,d',c)}{c^{\frac{1}{2}}}J_{\rho-\frac{1}{2}}\left(\frac{4\pi\sqrt{m^2\Delta(\gamma)}}{c}\right) & \text{ if } m \neq 0, \\
2^{\rho-1} \Delta(\gamma)^{\frac{\rho}{2}} \sum_{0 < c \equiv 0 \pmod*{4}} \frac{T_0(d,d',c)}{c^{\rho}} & \text{ if } m = 0.
\end{cases}
\end{align*}
\end{lemma}

\begin{rmk}
Note that Duke, Imamo\={g}lu, T\'{o}th and Matsusaka \cite[p.\ 10]{matsu2} use different notations regarding the cycle integral. This is caused by a different choice of generators of $\Gamma_Q$. Let $Q = \left[a,b,c\right]$ be a given primitive quadratic form, and let $t$, $u \in \N$ be the smallest solutions to Pell's equation $t^2-\Delta(Q)u^2 = 4$. Then, the authors of \cite{duimto11} employed the generator $\eta_Q = \pm \left(\begin{smallmatrix} \frac{t+bu}{2} & cu \\ -au & \frac{t-bu}{2}\end{smallmatrix}\right)$, while Matsusaka  works with $\Gamma_{Q_{\gamma}} = \pm\langle\gamma\rangle$. The associated quadratic form $Q_{\eta_Q}$ to $\eta_Q$ is given by $[-au,-bu,-cu] = -uQ$.
\end{rmk}

Recall that $T_m(d,d',c) = T_{-m}(d,d',c)$ by Lemma \ref{lem:qsymm}. We deduce from Lemma \ref{lem:laplace} and \ref{lem:duimto} that the Fourier coefficients corresponding to $m \neq 0$ are all regular at $s=0$, and vanish for every $m < 0$. To inspect the coefficient corresponding to $m=0$, we separate the cases $k \geq 4$ even and $k=2$.
\subsubsection{The case $k \geq 4$ even}
If $k \geq 4$ is even, then $G_0(\tau,\rho)$ is regular at $\rho = \frac{k}{2}$. Hence, the Fourier coefficient corresponding to $m = 0$ vanishes at $s=0$ by Lemma \ref{lem:laplace}. In other words, $\Ec_{k,\gamma}(\tau,0)$ is holomorphic and vanishes at the cusp. Thus,
\begin{align*}
\Ec_{k,\gamma}(\tau,0) =  \frac{(-1)^{\frac{k}{2}} 2^{\frac{k+3}{2}} \pi^{\frac{k}{2}+1}}{\Delta(\gamma)^{\frac{k-1}{4}}\Gamma\left(\frac{k}{2}\right)} \sum_{m \geq 1} m^{\frac{k-1}{2}} \sum_{a \geq 1} \frac{T_m(d,d',4a)}{2\sqrt{a}} J_{\frac{k-1}{2}}\left(\frac{\pi m \sqrt{\Delta(\gamma)}}{a}\right) \ q^m,
\end{align*}
and by Lemma \ref{lem:duimto}, we ultimately obtain that $\Ec_{k,\gamma}(\tau,0)$ equals
\begin{align*}
\frac{(-1)^{\frac{k}{2}}2\pi^{\frac{k}{2}}}{\Delta(\gamma)^{\frac{k}{4}}\Gamma\left(\frac{k}{4}\right)^2} \sum_{m\geq 1} m^{\frac{k}{2}-1} \sum_{Q \in \Qc\left(\Delta(\gamma)\right)_{\sim}} \chi_d(Q) \int_{\bsfrac{S_Q}{\Gamma_Q}} G_{-m}\left(w,\frac{k}{2}\right) \frac{\vt{dw}}{\Im{(w)}} \ q^m.
\end{align*}
This proves Theorem \ref{thm:var}.

\subsubsection{The case $k=2$}
We first suppose that $m \geq 1$. Then we define for $\Re{(s)} > 1$
\begin{align*}
j_m(\tau,s) \coloneqq G_{-m}(\tau,s) - \frac{2m^{1-s}\sigma_{2s-1}(m)}{\pi^{-s-\frac{1}{2}}\Gamma\left(s+\frac{1}{2}\right)\zeta(2s-1)}G_0(\tau,s),
\end{align*}
which has an analytic continuation up to $\Re{(s)} > \frac{1}{2}$ (cf.\ \cite[p.\ 970]{duimto11}). On one hand, the left hand side specializes at $s = 1$ to (cf.\ \cite[eq. (4.11)]{duimto11})
\begin{align*}
j_m(\tau,1) = j_m(\tau) = q^{-m} + O(q),
\end{align*}
which we encountered during the weight $2$ elliptic case already. On the other hand,
\begin{align*}
\lim_{s \to 1} (s-1)G_0(\tau,s) = \frac{3}{\pi}, \qquad \lim_{s \to 1} (s-1)\zeta(2s-1) = \frac{1}{2},
\end{align*}
from which we infer (see \cite[p.\ 1545]{andu} as well)
\begin{align*}
\lim_{s \to 1} \frac{2m^{1-s}\sigma_{2s-1}(m)}{\pi^{-s-\frac{1}{2}}\Gamma\left(s+\frac{1}{2}\right)\zeta(2s-1)}G_0(\tau,s) = 24\sigma_1(m).
\end{align*}
Combining, we arrive at the Fourier coefficients
\begin{align*}
\frac{-2}{\Delta(\gamma)^{\frac{1}{2}}} \sum_{m \geq 1} \sum_{Q \in \Qc\left(\Delta(\gamma)\right)_{\sim}} \chi_d(Q) \int_{\bsfrac{S_Q}{\Gamma_Q}} \left(j_m(w) + 24\sigma_1(m)\right)\frac{\vt{dw}}{\Im{(w)}} \ q^m.
\end{align*}

Secondly, we consider the case $m=0$, namely the Fourier coefficient
\begin{align*}
i\sum_{a \geq 1} \frac{T_0(d,d',4a)}{a^{s+1}} \int_{v-i\infty}^{v+i\infty} \frac{v^s}{\left(t^2+\lambda^2\right) \vt{t^2+\lambda^2}^{s}} dt.
\end{align*}
By Lemma \ref{lem:duimto}, the pole of
\begin{align*}
\sum_{a \geq 1} \frac{T_0(d,d',4a)}{a^{\rho}} = \frac{2\Gamma\left(\rho\right)}{\Delta(\gamma)^{\frac{\rho}{2}}\Gamma\left(\frac{\rho}{2}\right)^2}\sum_{Q \in \Qc\left(\Delta(\gamma)\right)_{\sim}} \chi_d(Q) \int_{\bsfrac{S_Q}{\Gamma_Q}} G_{0}\left(w,\rho\right) \frac{\vt{dw}}{\Im(w)}.
\end{align*}
at $\rho=1$ is simple, while
\begin{align*}
f(\rho) \coloneqq \int_{v-i\infty}^{v+i\infty} \frac{iv^{\rho-1}}{\left(t^2+\lambda^2\right) \vt{t^2+\lambda^2}^{\rho-1}} dt
\end{align*}
has a zero at $\rho = 1$ by Lemma \ref{lem:laplace}. We perform a Taylor expansion of $f$ around $1$, and note that only the term $(\rho-1)\frac{df}{d\rho}(1)$ survives in the limit $\rho \to 1$. We compute
\begin{align*}
\frac{df}{d\rho}(1) &= i\int_{v-i\infty}^{v+i\infty} \frac{v^{\rho-1}\log\left(\frac{v}{\vt{t^2+\lambda^2}}\right)}{\vt{t^2+\lambda^2}^{\rho-1}\left(t^2+\lambda^2\right)} \Bigg\vert_{\rho=1} dt = i\int_{v-i\infty}^{v+i\infty} \frac{\log(v) - \log \left(\vt{t^2+\lambda^2}\right)}{t^2+\lambda^2} dt \\
&= -i\int_{v-i\infty}^{v+i\infty} \frac{\log \left(\vt{t^2+\lambda^2}\right)}{t^2+\lambda^2} dt = \int_{-\infty}^{\infty} \frac{\log \left(\vt{(v+it)^2+\lambda^2}\right)}{(v+it)^2+\lambda^2} dt.
\end{align*}
We expand the integrand around $\lambda = 0$, which yields
\begin{align*}
\frac{df}{d\rho}(1) &= \left[-\frac{1}{v}\arctan{\left(\frac{t}{v}\right)} + \frac{\log{\left(v^2+t^2\right)}+1}{t-iv}\right]_{-\infty}^{\infty} + O\left(\lambda^2\right) = -\frac{\pi}{v} + O\left(\lambda^2\right).
\end{align*}
Recalling the definition of $\lambda$, we express the error as $O\left(\frac{1}{a^2}\right)$. Thus, the additional sums over $a$ caused by the expansion with respect to $\lambda$ are all regular at $\rho=1$ due to the proof of Lemma \ref{lem:fconv}. Hence, letting $\rho \to 1$ annihilates all error terms. Invoking  Lemma \ref{lem:duimto} and the residue of $G_0(w,\rho)$ at $\rho=1$ once more, we obtain 
\begin{align*}
\frac{-2}{\Delta(\gamma)^{\frac{1}{2}}}\sum_{Q \in \Qc\left(\Delta(\gamma)\right)_{\sim}} \chi_d(Q) \int_{\bsfrac{S_Q}{\Gamma_Q}} \frac{3}{\pi v} \frac{\vt{dw}}{\Im{(w)}}.
\end{align*}

In conclusion, we have shown that the analytic continuation of $\Ec_{2,\gamma}(\tau,s)$ to $s=0$ exists, and indeed equals the shape which Matsusaka conjectured in \cite[eq.\ (2.12)]{matsu2} for an individual hyperbolic Eisenstein series $E_{k,\gamma}(\tau,s)$. This proves Theorem \ref{thm:main}.

\begin{rmk}
In \cite[eq.\ (16)]{duimto10}, Duke, Imamo\={g}lu, T\'{o}th related Parson's Poincar{\'e} series \cite{parson} with the generating function $F(z,Q)$ of cycle integrals of functions $f_{k,m}$, where the functions $f_{k,m}$ generalize the functions $j_m$ to any even weight $k$ (compare \cite[Theorem 1, eq.\ (8)]{duimto10}. Since we realized the Fourier coefficients as cycle integrals of $G_{-m}\left(w,\frac{k}{2}\right)$, there might be a relation between them.
\end{rmk}

\bibliographystyle{amsplain}

\end{document}